\documentclass[11pt,a4paper]{article}
\usepackage{amscd}
\usepackage{amsmath}
\usepackage{amsthm}

\newtheorem{thm}{Theorem}[section]
\newtheorem{cor}[thm]{Corollary}
\newtheorem{lem}[thm]{Lemma}
\newtheorem{exm}{Example}
\newtheorem{prop}[thm]{Proposition}
\newtheorem{defn}[thm]{Definition}
\newtheorem{rem}[thm]{Remark}
\newtheorem{ass}[thm]{Assumption}
\def\D{\mathcal{D}}
\def\Z{\mathcal{Z}}
\def\C{\mathcal{C}}

\begin{document}

\begin{center}
{\Large \bf
Triangulated quotient categories}
\bigskip

{\large Yu Liu and Bin Zhu\footnote{Supported partly by the NSF of
China (Grants 11071133, 11131001) and by in part Doctoral Program
Foundation of Institute of Higher Education (2009).}}
\bigskip

{\small
\begin{tabular}{cc}
Department of Mathematical Sciences
\\
Tsinghua University
\\
100084 Beijing, P. R. China
\\
{\footnotesize E-mail: liuyu04@mails.tsinghua.edu.cn;
bzhu@math.tsinghua.edu.cn}
\end{tabular}
}
\bigskip


\end{center}

\def\s{\stackrel}
\def\Hom{\mbox{Hom}}

\begin{abstract}
 A notion of mutation of subcategories in a right triangulated category is defined in this
     paper. When $(\Z,\Z)$ is a $\D-$mutation pair in a right triangulated category $\mathcal{C}$,
     the quotient category $\Z/\D$ carries naturally a right triangulated structure.   Moreover,
     if the right triangulated category satisfies some reasonable conditions, then the right triangulated quotient
     category $\Z/\D$
     becomes a triangulated category. When $\C$ is triangulated, our result unifies the constructions of the quotient triangulated
     categories by Iyama-Yoshino and by J{\o}rgensen respectively.

\end{abstract}

\begin{center}

\end{center}
\medskip


\section{Introduction}

Triangulated categories are important structure in algebra and
geometry. There are two major ways to produce triangulated
categories: forming homotopy or derived categories of abelian
categories; and forming the stable category of Frobenius categories
[H] [BR] [ASS] [N].

Among the surprises produced by the recent study on cluster algebras
and cluster tilting theory is the possibility to define the notion
of mutation in a triangulated category by Iyama-Yoshino [IY], which
is a generalization of mutation of cluster tilting objects in
cluster categories [BMRRT][KR][KZ]. The latter models the mutation
of clusters of acyclic cluster algebras [FZ, K1, K2].  As one of
main results in [IY], Iyama and Yoshino proved that if $\cal
D\subseteq \cal Z$ are subcategories of a triangulated category
$\cal C$, and if $(\cal Z, Z)$ is a ${\cal D}-$mutation pair, where
$\mathcal{D}$ is rigid, i.e. $Ext^1(\mathcal{D},\mathcal{D})=0$,
then the quotient category $\cal Z / \cal D$ is a triangulated
category. Soon later, J{\o}rgensen [J] gave a similar construction
of triangulated category by quotient category in another manner. He
proved that if $\cal X$ is a functorially finite subcategory of a
 triangulated category $\cal C$ with Auslander-Reiten translate $\tau$,
 and if $\cal X$ satisfies the equation $\tau
\cal X=\cal X$, then the quotient
 category $\cal C / \cal X$ is a triangulated category.  Recently
 the authors define the mutation of torsion pairs in a triangulated
 category and give its geometric interpretation in [ZZ].

The aim of the paper is to unify these two constructions of the quotient
triangulated categories by Iyama-Yoshino in [IY] and by J{\o}rgensen
in [J] respectively. We define the notion of $\D-$mutation without
the assumption that $\D$ is rigid (compare [IY]). This generalizes the notion of
$\D-$mutation defined in [IY] where $\D$ is assumed rigid. If
$\D$ satisfies the condition of Theorem 2.3 in [J], i.e. $\cal D$ is a
functorially finite subcategory of a triangulated category $\cal C$
which satisfies the equation $\tau \cal X=\cal X$, then $(\C, \C)$
is a $\D-$mutation in our sense. Finally we prove that if $\cal
D\subseteq \cal Z$ are subcategories of a triangulated category
$\cal C$, and $(\cal Z, Z)$ is a ${\cal D}-$mutation pair, then the
quotient category $\cal Z / \cal D$ is a triangulated category.
\bigskip

Actually, our setting is right triangulated categories which were
defined and studied by Beligiannis, Assem and N.Marmaridis in [AB],
[ABM].

 The paper is organized as follows: In Section 2, we recall the definition of right triangulated category from
 [ABM], and define the notion of $\D-$mutation pair in it. We give some basic properties of right triangulated
category and of its quotient categories which are needed in the
proof of our main theorem. In Section 3, we state and prove the main
results of this paper.

\bigskip

\section{Right triangulated category}

Throughout the paper, all the subcategories of a category are full
subcategories and closed under isomorphisms. We recall some basics
on right triangulated categories from [AB], [ABM].

\begin{defn}
Let $\cal C$ be an additive category  and T an additive endofunctor of
$\cal C$. A sextuple $(A,B,C,f,g,h)$ in $\cal C$ is given by objects
$A,B,C \in \cal C$ and morphisms $f: A\rightarrow B$, $ g:
B\rightarrow C$ and $h: C\rightarrow TA$. A more suggestive notation
of sextuple is
$$A\s{f}{\rightarrow}B\s{g}{\rightarrow}C\s{h}{\rightarrow}TA.$$
A morphism from sextuples $(A,B,C,f,g,h)$ to $(A',B',C',f',g',h')$
is a triple $(a,b,c)$ of morphisms such that the following diagram
commutes:
\[ \begin{CD}
A@>f>>B @>g>>C@>h>>TA\\
@VVaV@VVb V @VVcV @VVTa V \\
A'@>f'>>B'@>g'>>C'@>h'>>TA'.
    \end{CD} \]
If in addition  $a,b$ and $c$ are isomorphisms in $\cal C$, the
morphism is then called an isomorphism of sextuples.
\end{defn}

A class $\sum$ of sextuples in $\cal C$ is called a right
triangulation of $\cal C$ if the following conditions $TR(0)-TR(5)$ are satisfied.
The elements of $\sum$ are then called right triangles in $\cal C$,
and the tripe $({\cal{ C}}, T,\sum)$ is called a right triangulated
category, or simply $\C$ is called a right triangulated category. The functor $T$ is called the shift functor of the right
triangulated category $\C$.

 Thus if $T:\C\rightarrow \C$ is an
equivalence, the right triangulated category $\cal C$ is a
triangulated category. In this case, right triangles in $\C$ are
called triangles [H].
\medskip

$TR(0).$ $\sum$ is closed under isomorphisms.\\
$TR(1).$ For any $A\in \cal C$, $$0\s{0}{\rightarrow}A\s{1}{\rightarrow}A\s{0}{\rightarrow}0$$ is a right triangle.\\
$TR(2).$ Any morphism $f: A\rightarrow B$ in $\cal C$ can be
extended to a right triangle $$A\s{f}{\rightarrow}B\s{g}{\rightarrow}C\s{h}{\rightarrow}TA.$$\\
$TR(3).$ If
$$A\s{f}{\rightarrow}B\s{g}{\rightarrow}C\s{h}{\rightarrow}TA$$ is a
right triangle, then
$$B\s{g}{\rightarrow}C\s{h}{\rightarrow}TA\s{-Tf}{\rightarrow}TB$$ is a right triangle.\\
$TR(4).$ Given a commutative diagram where the rows are right
triangles as follow:
$$ \begin{array}{cccclcl}
A&\s{f}{\rightarrow}&B&\s{g}{\rightarrow}&C&\s{h}{\rightarrow}&TA\\
\downarrow a&&\downarrow b&&\ &&\downarrow Ta\\
A'&\s{f'}{\rightarrow}&B'&\s{g'}{\rightarrow}&C'&\s{h'}{\rightarrow}&TA',\\
\end{array}$$
there exists a morphism $(a,b,c)$ from the first right triangle to the second.\\
$TR(5).$ (Octahedral axiom) Consider right triangles
$X\s{a}{\rightarrow}Y\s{b}{\rightarrow}Z\s{c}{\rightarrow}TX$,
$Y\s{d}{\rightarrow}U\s{e}{\rightarrow}V\s{f}{\rightarrow}TY$ and
$X\s{da}{\rightarrow}U\s{g}{\rightarrow}W\s{h}{\rightarrow}TX$. Then
there exist morphisms $l:Z\rightarrow W$ and $i:W\rightarrow V$ such
that the following diagrams commute and the third column in first
diagram is a right triangle.
$$ \begin{array}{ccccclclclcl}
X&\s{a}{\rightarrow}&Y&\s{b}{\rightarrow}&Z&\s{c}{\rightarrow}&TX\\
\| &&\downarrow d&&\downarrow l&&\| &&\\
X&\s{da}{\rightarrow}&U&\s{g}{\rightarrow}&W&\s{h}{\rightarrow}&TX\\
\ &&\downarrow e&&\downarrow i&&\ &&\\
&&V&=&V&&\\
\ &&\downarrow f&&\downarrow &&\ &&\\
&&TY&\s{Tb}{\rightarrow}&TZ&&\\
\end{array}$$
\[ \begin{CD}
X@>da>>U @>g>>W@>h>>TX\\
@VVaV@VV1 V @VViV @VVTa V \\
Y@>d>>U@>e>>V@>f>>TY
    \end{CD} \]

\bigskip

\begin{prop}
Let $\cal C$ be a right triangulated category, $A\s{f}{\rightarrow}B\s{g}{\rightarrow}C\s{h}{\rightarrow}TA$ be
a right triangle and E an object in $\cal C$. Then we have the following long exact sequence:
$$Hom_{\C}(A,E)\s{\circ f}{\leftarrow}Hom_{\C}(B,E)\s{\circ g}{\leftarrow}Hom_{\C}(C,E)\s{\circ h}{\leftarrow}Hom_{\C}(TA,E){\leftarrow}\cdot\cdot\cdot$$
\end{prop}

\begin{proof}
It is enough to show that $$Hom(A,E)\s{\circ
f}{\leftarrow}Hom(B,E)\s{\circ g}{\leftarrow}Hom(C,E)$$ is exact. By
$TR(1)$,$TR(3)$ and $TR(4)$, we obtain that the following
commutative diagrams of right triangles:
\[ \begin{CD}
A@>1>>A @>0>>0@>0>>TA\\
@VV1V@VVf V @VV0V @VV1 V \\
A@>f>>B@>g>>C@>h>>TA
    \end{CD} \]
Hence $gf=0$, i.e. $im(\circ g)\subseteq ker(\circ f)$. If we have
$if=0$, $i \in Hom_{\cal C}(B,E)$, by $TR(1)$, we obtain that the
following commutative diagrams of right triangles:
$$ \begin{array}{cccclcl}
A&\s{f}{\rightarrow}&B&\s{g}{\rightarrow}&C&\s{h}{\rightarrow}&TA\\
\downarrow 0&&\downarrow i&&\ &&\downarrow 0\\
0&\s{0}{\rightarrow}&E&\s{1}{\rightarrow}&E&\s{0}{\rightarrow}&0\\
\end{array}$$
By $TR(4)$, there exists $j:C\rightarrow E$ such that $i=jg$, i.e.
$ker(\circ f)\subseteq im(\circ g)$. Then $ker(\circ f)= im(\circ
g)$.
\end{proof}

\begin{defn}
A  subcategory  $\cal Z$ of a right triangulated category $\cal C$
is called extension-closed if for any right triangle
$A\s{f}{\rightarrow}B\s{g}{\rightarrow}C\s{h}{\rightarrow}TA$ with
$A,C \in \cal Z$, then we get $B \in \cal Z$.
\end{defn}

\begin{defn}
Let $\cal D$ be a subcategory of a right triangulated category $\cal
C$. A morphism $f:A\rightarrow B$ in $\cal C$ is called ${\cal
D}-epic$, if for any $D \in \cal D$, we have that
$$Hom_{\cal C}(D,A)\s{f\circ}{\rightarrow}\Hom_{\cal C}(D,B)\s{}{\rightarrow}0$$
is exact. Dually, a morphism $f:A\rightarrow B$ in $\cal C$ is
called ${\cal D}-monic$, if for any $D \in \cal D$,
$$Hom_{\cal C}(B,D)\s{\circ f}{\rightarrow}Hom_{\cal C}(A,D)\s{}{\rightarrow}0$$
is exact.
\end{defn}
\bigskip

\begin{defn}
Let $\cal D$ be a subcategory of a right triangulated category $\cal C$. A morphism
$f:A\rightarrow B$ in $\cal C$ is called a right ${\cal
D}-approximation$ of B if $A \in \cal D$ and f is a ${\cal D}-epic$.
Dually, a morphism $f:A\rightarrow B$ in $\cal C$ is called a left ${\cal D}-approximation$ of A if $B \in \cal D$ and f is a ${\cal D}-monic$.
\end{defn}
\bigskip

Now we assume that a right triangulated category $\cal C$ is
Krull-Schmidt, i.e. any object is isomorphic to a finite direct sum
of objects whose endomorphism rings are local. When we say that
$\cal D$ is a subcategory of $\cal C$, we always mean that $\cal D$
is full and is closed under isomorphisms, direct sums and direct
summands.
\bigskip

The notion of $\D-$mutation of subcategories was defined in [IY] for
a rigid subcategory $\D$. This notion generalizes the mutation of
cluster tilting objects in cluster categories[BMRRT] which was motivated by modeling the mutation of clusters of cluster
algebras [FZ]. In the following we recall the notion of
$\D-$mutation of subcategories from [IY], but we don't assume that
$\D$ is rigid here.

\begin{defn}
Let $\cal X,Y,D$ be subcategories of a right triangulated category
 $\cal C$, and assume $\cal D \subseteq \cal X,$ and $ \cal D \subseteq Y$.
The subcategory $\mu^{-1}(\cal X; \cal D)$ is defined as the
subcategory consisting of objects $Y \in \cal C$ such that $Y\in \D$
or there exists a
 right triangle
$$X\s{f}{\rightarrow}D\s{g}{\rightarrow}Y\s{h}{\rightarrow}TX,$$
where $X \in \cal X$, $D \in \cal D$,  $f$ is a left ${\cal D}-approximation$ and $g$ is a right
${\cal D}-approximation$. Dually, for $\cal Y$, the subcategory
 $\mu(\cal Y; \cal D)$ is defined as the subcategory consisting of objects
 $X \in \cal C$ such that $X\in \D$ or there exists a right triangle
$$X\s{f}{\rightarrow}D\s{g}{\rightarrow}Y\s{h}{\rightarrow}TX,$$
where $Y \in \cal Y$, $D \in \cal D$, $f$ is a left ${\cal D}-approximation$ and $g$ is a right
${\cal D}-approximation$.

  A pair $(\cal X,Y)$ of subcategories of $\cal C$ is called a ${\cal
D}-mutation$ pair if $\mu^{-1}(\cal {X}; \cal {D})=\cal {Y}$ and
$\mu(\cal Y; \cal D)=\cal {X}$.
\end{defn}

\begin{defn}
Let $\cal D$ be a subcategory of a right triangulated category $\cal C$. We denote by $[{\cal
D}](X,Y)$ the subgroup of $Hom_{\cal C}(X,Y)$ consisting of
morphisms which factor through an object in $\cal D$. We say that
$\cal D$ is factor-through-epic if for any morphism $f \in [{T^n
\cal D}](TX,TY)$ with $n>0$, there exists a morphism $f' \in [{T^{n-1}\cal D}](X,Y)$
such that $Tf'=f$.
\end{defn}

\begin{rem} The zero subcategory $\cal D$$=0$ is factor-through-epic. Another typical case is that: when $\cal C$ is a
triangulated category, then any subcategory $D$ is
factor-through-epic. \end{rem}

\begin{lem}
For any two objects $A,B$ of a right triangulated category $\cal C$,
$$A\s{i_A}{\rightarrow}A\oplus B\s{p_B}{\rightarrow}B\s{0}{\rightarrow}TA$$
is a right triangle, where $i_A$ is a section and $p_B$ is a
retraction.
\end{lem}

\begin{proof}
By $TR(2)$, we can extend the section $A\s{i_A}{\rightarrow}A\oplus
B$ into a right triangle:
$$(*):  \  \  A\s{i_A}{\rightarrow}A\oplus B\s{p}{\rightarrow}C\s{c}{\rightarrow}TA.$$
By $TR(1)$, $TR(3)$ and $TR(4)$, the following diagram where the rows
are right triangles commutes:
\[ \begin{CD}
A@>i_A>>A\oplus B @>p>>C@>c>>TA\\
@VV1V@VVp_A V @VV0V @VV1 V \\
A@>1>>A@>0>>0@>0>>TA
    \end{CD} \]
It follows that $c=0$. Since $p_B$$\circ$$i_A$$=0$, by Prop 2.2,
there exists a morphism $f: C\rightarrow B$ such that $p_B$$=fp$.
From the following diagram:
$$ \begin{array}{cccclclcl}
&&B&\s{f}{\leftarrow}&C&&\\
\ &&\downarrow i_B&&\| &&\ &&\\
A&\s{i_A}{\rightarrow}&A \oplus B&\s{p}{\rightarrow}&C&\s{0}{\rightarrow}&TA\\
\ &&\downarrow p_B&&\| &&\ &&\\
&&B&\s{f}{\leftarrow}&C&&\\
\end{array},$$
we know that $(1_C-p \circ i_B \circ f) \circ p=p \circ (1_{A\oplus
B}-i_B \circ f \circ p)=p \circ i_A \circ p_A=0$. By Prop 2.2,
$1_C-p \circ i_B \circ f$ factors through $C\s{0}{\rightarrow}TA$,
hence $1_C-p \circ i_B \circ f=0$, and $1_C=p \circ i_B \circ f$. We
also have $f \circ (p \circ i_B)=(f \circ p) \circ i_B=1_B$.
 Then $f$ is an isomorphism. Thus  the sextuple $A\s{i_A}{\rightarrow}A\oplus B\s{p_B}{\rightarrow}B\s{0}{\rightarrow}TA$
 is isomorphic to the right triangle $(*)$ under $(1_A, 1_{A\oplus B}, f)$. Hence it is a right triangle.
\end{proof}

\begin{ass} From now on to the end of the paper, we assume that
$\cal C$ is a right triangulated category  and satisfies: If
$B\s{g}{\rightarrow}C\s{h}{\rightarrow}TA\s{-Tf}{\rightarrow}TB$ is
a right triangle, then
$A\s{f}{\rightarrow}B\s{g}{\rightarrow}C\s{h}{\rightarrow}TA$ is a
right triangle.\end{ass}

\begin{lem} Let $\C$ be a right triangulated category satisfying Assumption 2.10. Then the shift functor $T$ is faithful:\end{lem}

{\bf Proof.} By Lemma 2.9, for any objects $A$ and $B$, there exists a right triangle:\\
$$B\s{i_B}{\rightarrow}B\oplus TA\s{p_{TA}}{\rightarrow}TA\s{0}{\rightarrow}TB.$$
For any morphism $f: A\rightarrow B$, if $Tf=0$, by the Assumption
2.10, we have the right triangle
$$A\s{f}{\rightarrow}B\s{i_B}{\rightarrow}B\oplus
TA\s{p_{TA}}{\rightarrow}TA.$$
By Proposition 2.2, ${i_B}f=0$, but $i_B$ is a monomorphism, thus $f=0$.\\

\begin{rem} Triangulated categories satisfy the Assumption 2.10.
There are right triangulated categories satisfying the assumption,
see Example 4 in Section 3.5. \end{rem}

\begin{prop} Let $(a,b,c)$ be a morphism of right triangles in a
right triangulated category $\C$:
\[ \begin{CD}
A@>f>>B @>g>>C@>h>>TA\\
@VVaV@VVb V @VVcV @VVTa V \\
A'@>f'>>B'@>g'>>C'@>h'>>TA'.
    \end{CD} \]
If $a$ and $b$ are isomorphisms, so is $c$.
\end{prop}

\begin{proof}
By applying the cohomological functor $Hom_{\C}(-,X)$ to the commutative
diagram above, we obtain the following commutative diagram which has
exact rows:
\[ \begin{CD}
Hom_{\C}(TB,X)@>>>Hom_{\C}(TA,X)@>>>Hom_{\C}(C,X)@>>>Hom_{\C}(B,X)@>>>Hom_{\C}(A,X)\\
@VVHom_{\C}(Tb,X)V@VVHom_{\C}(Ta,X) V @VVHom_{\C}(c,X)V @VVHom_{\C}(b,X) V@VVHom_{\C}(a,X) V\\
Hom_{\C}(TB',X)@>>>Hom_{\C}(TA',X)@>>>Hom_{\C}(C',X)@>>>Hom_{\C}(B',X)@>>>Hom_{\C}(A',X)
    \end{CD} \]
By Snake-Lemma, $Hom_{\C}(c,X):Hom_{\C}(C,X)\rightarrow
Hom_{\C}(C',X)$ is an isomorphism for any $X \in C$, hence
$Hom_{\C}(c,-):Hom_{\C}(C',-)\rightarrow Hom_{\C}(C,-)$ is a
functorial isomorphism. By Yoneda Lemma, $c$ is an isomorphism.
\end{proof}

\bigskip

\section{Quotient categories of a right triangulated category}

\subsection{Basics on quotient categories}

\begin{defn}
Let $\cal D\subset \mathcal{Z} $ be subcategories of a category
$\cal C$. We denote by $\cal Z$$/$$\cal D$ the category whose
objects are objects of $\cal Z$ and whose morphisms are elements of
$Hom_{\cal Z}(X,Y)/[{\cal D}](X,Y)$ for $X,Y \in \cal Z$. Such
category is called the quotient category of $\cal Z$ by $\cal D$.
For any morphism  $f: X\rightarrow Y$ in $\cal Z$, we denote by
$\overline{f}$ the image of $f$ under the natural quotient functor
$\Z\rightarrow \Z/$$\cal D$.
\end{defn}

\begin{lem}
Let $\cal D$$\subseteq \cal Z$ be subcategories of a right
triangulated category $\cal C$ and $\cal D$ be factor-through-epic.
Consider the following commutative diagram:
\[ \begin{CD}
M@>\alpha>>D @>\beta>>S@>\gamma>>TM\\
@VVxV@VVy V @VVzV @VVTx V \\
X@>f>>Y@>g>>Z@>h>>TX,
    \end{CD} \]
where rows are right triangles, $D \in \cal D$ and $g$ is a ${\cal
D}-epic$. If $\overline{z}=0$ in the quotient category $\Z/\D$, then
$\overline{x}=0$.
\end{lem}

\begin{proof}
 $\overline{z}=0$ means that $z$ factors through an object $D' \in \cal
D$. Since $g$ is a ${\cal D}-epic$, we have the following
commutative diagram:
$$ \begin{array}{ccclcl}
D'&\s{1}{\rightarrow}&D'&\s{a}{\leftarrow}&S\\
\downarrow c&&\downarrow b&&\downarrow z\\
Y&\s{g}{\rightarrow}&Z&\s{1}{\rightarrow}&Z.\\
\end{array}$$
Hence $z=ba=gca$. Then $Tx \circ \gamma=hz=hgca=0$. By Proposition
2.2, there exists  $\nu :TD\rightarrow TX$ which makes the diagram
\[ \begin{CD}
TM@>-T\alpha>>TD @>-T\beta>>TS@>-T\gamma>>T^2M\\
@VVTxV@VV\nu V @VV0V @VVT^2xV \\
TX@>1>>TX@>0>>0@>0>>T^2X
    \end{CD} \]
commutative. Since $\cal D$ is factor-through-epic, there exists
${\nu}':D\rightarrow X$ such that $\nu=T{\nu}'$. Hence $Tx=-T\nu$$'
T\alpha$, which forces $x=-\nu$$' \alpha$, hence $\overline{x}=0$.
\end{proof}

\begin{lem}
Let $\cal D$ be a subcategory of a right triangulated category $\cal
C$ which is
 factor-through-epic. Consider the right triangle:
\[ \begin{CD}
A'@>f'>>D@>g'>>C'@>h'>>TA'
    \end{CD} \]
where $D \in \cal D$. If $h' \circ c=0$, for $c:C\rightarrow C'$,
then we can find $d \in Hom_{\C}(C,D)$ such that $c=g'd$.
\end{lem}

\begin{proof}
Since $h' \circ c=0$, there exists a commutative diagram:
\[ \begin{CD}
C@>0>>0 @>0>>TC@>-1>>TC\\
@VVcV@VV0 V @VVd_1V @VVTc V \\
C'@>h'>>TA'@>-Tf'>>TD@>-Tg'>>TC',
    \end{CD} \]
where the rows are right triangles, $d_1$ exists by $(TR4)$. Since
$\cal D$ is factor-through-epic, there is a morphism $d \in
Hom_{\C}(C,D)$ such that $d_1=Td$. Hence $Tc=Tg'Td=T(g'd)$. Then
$c=g'd$ by Lemma 2.11.
\end{proof}

\begin{lem}
Let $\cal D \subset \mathcal{Z}$ be subcategories of a right
triangulated category $\cal C$ and $\cal D$ be factor-through-epic.
Consider the following commutative diagram:
\[ \begin{CD}
A@>f>>B @>g>>C@>h>>TA\\
@VVaV@VVb V @VVcV @VVTa V \\
A'@>f'>>D@>g'>>C'@>h'>>TA',
    \end{CD} \]
where the rows are right triangles, $D \in \cal D$ and $f$ a left
$\cal D$$-monic$. If $\overline{a}=0$ in the quotient category
$\Z/\D$, then $\overline{c}=0$.
\end{lem}

\begin{proof}
By the condition $\overline{a}=0$, we have that $a$ factors through
an object $D_1 \in \cal D$. Since $f$ is a left $\cal D$$-monic$, we
have the following commutative diagram:
$$ \begin{array}{ccclcl}
A&\s{1}{\rightarrow}&A&\s{a}{\rightarrow}&A'\\
\downarrow f&&\downarrow d_1&&\downarrow 1\\
B&\s{d}{\rightarrow}&D_1&\s{b}{\rightarrow}&A'.\\
\end{array}$$
Hence $a=bd_1f$, and $h'c=Ta \circ h=T(bd_1f)h=T(bd_1)T(f)h=0$. By
Lemma 3.3, $c$ factors through $D$, so we get $\overline{c}=0$.
\end{proof}
\bigskip

\subsection{Right triangles on the quotient category}

In this subsection, $\C$ denotes a right triangulated category
satisfying Assumption 2.10.  Assume that $\cal D$$\subseteq \cal Z$
are subcategories of $\cal C$, $\cal D$ is factor-through-epic,
$\cal Z$ is extension-closed and satisfies $\cal Z=\mu (\cal Z;\cal
D)$. Then for any object $M \in \cal Z$, there exists a right
triangle
$$M\s{\alpha _M}{\rightarrow}D_M\s{\beta _M}{\rightarrow}\sigma M\s{\gamma _M}{\rightarrow}TM$$
with $\sigma M \in \cal Z$ and $D_M \in \cal D$, and moreover
$\alpha _M$ is a
left $\cal D-$approximation and $\beta _M$ is a right $\cal D-$approximation.\\

\begin{defn}
Let $M\s{\mu }{\rightarrow}N\s{\gamma
}{\rightarrow}P\s{\varphi}{\rightarrow} TM$ be a right triangle in
$\C$, where $\mu $ is $\cal D$$-monic$ and $M,N,P \in \cal Z$. Then
there exists a commutative diagram where the rows are right
triangles:
\[ \begin{CD}
M@>\mu >>N @>\gamma >>P@>\varphi >>TM\\
@VV1V@VVn V @VV\pi V @VV1V \\
M@>\alpha _M>>D_M@>\beta _M>>\sigma M@>\gamma _M>>TM.
    \end{CD} \]

Then we have the following sextuple in the quotient category $\cal Z
/ \cal D$:
$$(*): M\s{\overline{\mu}}{\rightarrow}N\s{\overline{\gamma}}{\rightarrow}P\s{\overline{\pi}}{\rightarrow}\sigma
M$$

 We define the right triangles in $\cal Z / \cal D$ as the sextuples
  which are isomorphic to $(*).$

\end{defn}

\begin{rem} It is easy to prove that $\sigma M$ is unique
up to isomorphism in the quotient category $\cal Z / \cal D$.
 So for any $M \in \cal Z$, we fix a right triangle
$$M\s{\alpha _M}{\rightarrow}D_M\s{\beta _M}{\rightarrow}\sigma M\s{\gamma _M}{\rightarrow}TM$$
In particular, for any $M \in \cal D$, we fix a right triangle
$$M\s{1}{\rightarrow}M\s{0}{\rightarrow}0\s{0}{\rightarrow}TM.$$\end{rem}
\bigskip

For any morphism $\mu \in Hom_{\cal Z}(M,N)$, where $M,N\in \Z$,
there exist $g$ and ${\mu}'$ which make the following diagram
commutative.
\[ \begin{CD}
M@>\alpha _M >>D_M @>\beta _M >>\sigma M@>\gamma _M>>TM\\
@VV\mu V@VVg V @VV{\mu}' V @VVT\mu V \\
N@>\alpha _N >>D_N @>\beta _N >>\sigma N@>\gamma _N>>TN.
    \end{CD} \]

We define an endofunctor $\sigma :{\cal Z} / {\cal D}\rightarrow
{\cal Z} / {\cal D}$ as follows: $\sigma :M\mapsto \sigma (M)$,
$\overline \mu \mapsto \overline {\mu'}$.

\begin{prop}
$\sigma :{\cal Z} / {\cal D}\rightarrow {\cal Z} / {\cal D}$ is an
additive functor.
\end{prop}

\begin{proof}
One can easily check that $\sigma$ satisfies the definition of the additive functor.
 We only prove that $\sigma$ is well-defined. Now assume $\mu ,\mu_1 \in Hom_{\cal Z}(M,N)$,
 $\overline{\mu}=\overline{\mu _1}$. Then we also have the commutative
 diagram where the rows are right triangles (compare to the commutative diagram before the proposition):
\[ \begin{CD}
M@>\alpha _M >>D_M @>\beta _M >>\sigma M@>\gamma _M>>TM\\
@VV\mu _1V@VVg' V @VV{\mu _1}' V @VVT\mu_1 V \\
N@>\alpha _N >>D_N @>\beta _N >>\sigma N@>\gamma _N>>TN.
    \end{CD} \]
 Since $\overline{\mu} -\overline{\mu _1}=0$, by Lemma 3.4, we have
 that $\overline{\mu'} =\overline{\mu _1'}$
\end{proof}

\begin{lem} Assume that we have a commutative diagram where the rows are right triangles in $\C$:
\[ \begin{CD}
M@>\mu >>N @>\gamma >>P@>\varphi >>TM\\
@VVfV@VVg V @VVh V @VV TfV \\
M'@>{\mu}'>>N'@>{\gamma}'>>P'@>{\varphi}'>>TM',
    \end{CD} \]
 where $M,N,P,M',N',P' \in \cal Z$ and $\mu$ and $ {\mu}'$ are ${\D}-monic$. Then we have the following commutative diagram in
$\cal Z / \cal D:$
\[ \begin{CD}
M@>\overline \mu >>N @>\overline \gamma >>P@>\overline \pi>>\sigma M\\
@VV\overline fV@VV\overline g V @VV\overline h V @VV \sigma (\overline f)V \\
M'@>\overline {\mu}'>>N'@>\overline {\gamma}'>> P'@>\overline
{\pi}'>>\sigma M'
    \end{CD} \]
\end{lem}

\begin{proof}
Consider the following commutative diagrams where the rows are right
triangles in $\C$, and $\overline{f'}=\sigma (\overline f)$:
\[ \begin{CD}
M'@>{\mu}' >>N' @>{\gamma}' >>P'@>{\varphi}' >>TM'\\
@VV1V@VVn' V @VV{\pi}' V @VV V \\
M'@>\alpha _{M'}>>D_{M'}@>\beta _{M'}>>\sigma M'@>\gamma _{M'}>>TM'
    \end{CD} \]
and
\[ \begin{CD}
M@>\alpha _M >>D_M @>\beta _M >>\sigma M@>\gamma _M>>TM\\
@VVf V@VV V @VVf' V @VV TfV \\
M'@>\alpha _{M'} >>D_{M'} @>\beta _{M'} >>\sigma M'@>\gamma
_{M'}>>TM'.
    \end{CD} \]
 We have that
 $\gamma _{M'}(f'\pi-{\pi}'h)=Tf \circ \gamma _M \circ \pi-{\varphi}'h=Tf \circ \varphi -{\varphi}'h=0$.
  It follows from Lemma 3.3 that $\overline {f'\pi}=\overline {{\pi'}h}$. Then the following diagram commutes
\[ \begin{CD}
M@>\overline \mu >>N @>\overline \gamma >>P@>\overline \pi>>\sigma M\\
@VV\overline fV@VV\overline g V @VV\overline h V @VV \sigma (\overline f)V \\
M'@>\overline {\mu'}>>N'@>\overline {\gamma'}>> P'@>\overline
{\pi'}>>\sigma M'
    \end{CD} \]
\end{proof}

\subsection{Main theorem}

\begin{thm} Let $\C$ be a right triangulated category satisfying
Assumption 2.10. Assume that $\cal D$$\subseteq \cal Z$ are
subcategories of $\cal C$, $\cal D$ is factor-through-epic, $\cal Z$
is extension-closed and satisfies $\cal Z=\mu (\cal Z;\cal D)$. Then
the quotient category $\cal Z / \cal D$ forms a right triangulated
category with the additive functor $\sigma$ and the right triangles
defined in Def 3.5.
\end{thm}

\begin{proof}
We will check that the right triangles in $\cal Z / \cal D$ defined
in Definition 3.5 satisfies the axioms of right triangulated
categories (see Definition 2.1). It follows from the definition of
right triangles in $\Z/\D$ that $TR(0)$ is satisfied.

For $TR(1):$ The commutative diagram
\[ \begin{CD}
0@>0>>M @>1>>M@>0>>0\\
@VVV@VV V @VVV @VV V \\
0@>0>>0@>0>>0@>0>>0
    \end{CD} \]
shows that $0\s{0}{\rightarrow}M\s{\overline 1}{\rightarrow}M\s{0}{\rightarrow}0$ is a right triangle.\\
For $TR(2):$ Let $\overline {\mu}:M\rightarrow N$ be any morphism in
$\Z/\D$. We have a morphism ${\mu}'= \left(\begin{array}{cclcl}
{\mu}\\
{\alpha}_M
\end{array}\right)$:$M\rightarrow N\oplus D_M$ which is ${\cal D}-monic$, where $\alpha_M: D_M\rightarrow M$
is a left $\D-$approximation. Suppose that
$M\s{{\mu}'}{\rightarrow}N\oplus D_M\rightarrow P'\rightarrow TM$ is
a right triangle in $\C$ which contains the morphism $\mu'$ as a
part. It follows from Lemma 2.9 and the octahedral axiom that we
have the following commutative diagram where the first two rows and
the second column are right triangles:
$$ \begin{array}{ccccclclclcl}
M&\s{{\mu}'}{\rightarrow}&N\oplus D_M&\s{{\gamma}'}{\rightarrow}&P'&\rightarrow&TM\\
\| &&\downarrow p_{D_M}&&\downarrow &&\| &&\\
M&\s{\alpha _M}{\rightarrow}&D_M&\rightarrow&\sigma M&\rightarrow&TM\\
\ &&\downarrow &&\downarrow &&\ &&\\
&&TN&=&TN&&\\
\ &&\downarrow -Ti_N&&\downarrow &&\ &&\\
&&TN\oplus TD_M&\s{T{\gamma}'}{\rightarrow}&TP'&&\\
\end{array}$$
It also follows the third column is a right triangle. Then by the
Assumption 2.10, $N\s{{\gamma}'i_N}{\rightarrow}P'\rightarrow \sigma
M\rightarrow TN$ is a right triangle in $\C$, where $N,\sigma M \in
\cal Z$.  Since $\cal Z$ is extension-closed, we have that $P' \in
\cal Z$. Thus there is a right triangle $M\s{\overline \mu
}{\rightarrow}N\rightarrow P'\rightarrow \sigma (M)$ in  $\cal Z /
\cal D$ which contains
$\overline{\mu}$ as a part.\\
For $TR(3):$ Let $M\s{\overline \mu}{\rightarrow}N\s{\overline
\gamma}{\rightarrow}P\s{\overline \pi}{\rightarrow} \sigma M$ be a
right triangle in $\cal Z / \cal D$. We assume that it is induced by
the right triangle in $\C$:
$M\s{\mu}{\rightarrow}N\s{\gamma}{\rightarrow}P\s{\varphi}{\rightarrow}TM$.
Considering the commutative diagram in Def 3.5, we have $T\alpha _M
\circ \varphi=Tn \circ T\mu \circ \varphi=0$. By the octahedral
axiom and Lemma 2.9, we have the following commutative diagrams
$$ \begin{array}{ccccclclclcl}
P&\s{\varphi}{\rightarrow}&TM&\s{-T\mu}{\longrightarrow}&TN&\s{-T\gamma }{\longrightarrow}&TP\\
\| &&\big\downarrow &&\big\downarrow {\gamma}'&&\| &&\\
P&\s{0}{\rightarrow}&TD_M&\s{\left(\begin{array}{cclcl}
1\\
0
\end{array}\right)}{\longrightarrow}&TD_M\oplus TP&\s{\left(\begin{array}{cccl}
0&1\\
\end{array}\right)}{\longrightarrow}&TP\\
\ &&\big\downarrow &&\big\downarrow {\pi}'&&\ &&\\
&&T\sigma M&=&T\sigma M&&\\
\ &&\big\downarrow &&\big\downarrow &&\ &&\\
&&T^2M&\s{-T^2\mu}{\rightarrow}&T^2N&&\\
\end{array}$$
and
\[ \begin{CD}
P@>0>>TD_M @>\left(\begin{array}{cclcl}
1\\
0
\end{array}\right)>>TD_M\oplus TP@>{\left(\begin{array}{cccl}
0&1\\
\end{array}\right)}>>TP\\
@VV \varphi V@VV 1V @VV {\pi}'V @VVT\varphi V \\
TM@>-T{\alpha}_M>>TD_M@>-T{\beta}_M>>T\sigma M@>-T{\gamma}_M>>T^2M.
    \end{CD} \]
Write  ${\gamma}'$ as ${\gamma}'= \left(\begin{array}{cclcl}
{\gamma _1}\\
{\gamma _2}
\end{array}\right)$. Then $-T\gamma =
\left(\begin{array}{cccl}
0&1\\
\end{array}\right)$
$\left(\begin{array}{cclcl}
{\gamma _1}\\
{\gamma _2}
\end{array}\right)$$={\gamma _2}$.
Since ${\gamma _1} \in Hom_{\C}(TN,TD_M)$ and $\cal D$ is
factor-through-epic, there exists ${\gamma _1}'\in Hom_{\C}(N,D_M)$
such that $T{\gamma _1}'={\gamma _1}$. Write ${\pi}'$ as ${\pi}'=
\left(\begin{array}{cccl}
{\pi_1}'&{\pi _2}'\\
\end{array}\right)$, then $\left(\begin{array}{cccl}
{\pi_1}'&{\pi _2}'\\
\end{array}\right)$
$\left(\begin{array}{cclcl}
1\\
0
\end{array}\right)$$={\pi _1}'=-T\beta _M$, $-T\gamma _M
\left(\begin{array}{cccl}
{\pi_1}'&{\pi _2}'\\
\end{array}\right)$
$=\left(\begin{array}{cccl}
{-T\gamma _M \circ \pi_1}'&{-T\gamma _M \circ \pi _2}'\\
\end{array}\right)$
$=\left(\begin{array}{cccl}
0&{-T\gamma _M \circ \pi _2}'\\
\end{array}\right)$
$=\left(\begin{array}{cccl}
0&T\varphi\\
\end{array}\right)$. Then $-T\gamma _M \circ {\pi _2}'=T\varphi$. Since $T\gamma_M \circ ({\pi _2}' +T\pi)=0$,
by Lemma 3.3, ${\pi _2}' +T\pi$ factors through $TD_M$, there exists
$\pi _2$ such that $T\pi _2={\pi _2}' +T\pi$. By the commutative
diagram before Proposition 3.7, we have the following commutative
diagram where the rows are right triangles in $\C$:
\[ \begin{CD}
\sigma M@>-T\mu \circ \gamma _M>>TN@>{\gamma}'>>TD_M\oplus TP@>{\pi}'>>T\sigma M\\
@VV -{\mu}' V@VV 1V @VV V @VV-T{{\mu}'} V \\
\sigma N@>\gamma _N>>TN@>-T\alpha_N>>TD_N@>-T\beta_N>>T\sigma N
    \end{CD} \]
By the Assumption 2.10, we have the following commutative diagram
where the rows are right triangles:
\[ \begin{CD}
N@>{\left(\begin{array}{cclcl}
-{\gamma _1}'\\
\gamma
\end{array}\right)}>>D_M\oplus P@>{(\beta _M,\pi -\pi _2)}>>\sigma M@>-T\mu \circ \gamma _M>>TN\\
@VV 1V@VV V @VV -{\mu}'V @VV 1V \\
N@>\alpha _N >>D_N @>\beta _N >>\sigma N@>\gamma _N>>TN,
    \end{CD} \]
where $\left(\begin{array}{cclcl}
-{\gamma _1}'\\
\gamma
\end{array}\right)$ $\cal D$$-monic$. Then
$N\s{\overline \gamma}{\rightarrow}P\s{\overline
\pi}{\rightarrow}\sigma M\s{-\sigma (\overline \mu)=-\overline
{\mu'}}{\longrightarrow}\sigma N$
is a right triangle.\\
For $TR(4):$ Suppose there is a commutative diagram where the rows
are right triangles in $\cal Z / \cal D$:
$$ \begin{array}{llcccclcl}
&&A&\s{\overline f}{\rightarrow}&B&\s{\overline g}{\rightarrow}&C&\s{\overline h}{\rightarrow}&\sigma A\\
(*)&&\downarrow \overline a&&\downarrow \overline b&&\ &&\downarrow \sigma (\overline a)\\
&&A'&\s{\overline f'}{\rightarrow}&B'&\s{\overline g'}{\rightarrow}&C'&\s{\overline h'}{\rightarrow}&\sigma A'.\\
\end{array}$$ By the definition of right triangles of $\Z/\D$ in Def 3.5, there exists a(not necessarily
commutative)diagram where the rows are right triangles in $\C$:
$$ \begin{array}{llcccclcl}
&&A&\s{f}{\rightarrow}&B&\s{g}{\rightarrow}&C&\s{h}{\rightarrow}&TA\\
(**)&&\downarrow a&&\downarrow b&&\ &&\downarrow Ta\\
&&A'&\s{f'}{\rightarrow}&B'&\s{g'}{\rightarrow}&C'&\s{h'}{\rightarrow}&TA'.\\
\end{array}$$ Since $\overline {f'a}=\overline {bf}$, the morphism $bf -f'a$ factors
through an object $D \in \cal D$. Since $f$ is $\cal D$$-monic$, we
have the following commutative diagram
$$ \begin{array}{cccccl}
A&\s{1}{\leftarrow}&A&\s{1}{\rightarrow}&A&\\
\downarrow f&&\downarrow \alpha_3&&\downarrow & bf -f'a\\
B&\s{\alpha _1}{\rightarrow}&D&\s{\alpha _2}{\rightarrow}&B'&\\
\end{array}$$
Denote $\alpha _2 \alpha _1$ by $s$. We have $sf=bf -f'a$. After
replacing $b-s$ by $b$ in the diagram $(**)$,  whose images in
$\Z/\D$ are the same, we can assume that $bf =f'a$. Therefore there
exists a morphism $c: C\rightarrow C'$ such that $(a,b,c)$ is a
morphism of right triangles in $\C$.
It follows from Lemma 3.8 that $(\overline{a},\overline{b},\overline{c})$ is a morphism of triangles in $(*)$.\\
For $TR(5):$ Let $X\s{\overline a}{\rightarrow}Y\s{\overline
b}{\rightarrow}Z\s{\overline c}{\rightarrow}TX$, $Y\s{\overline
d}{\rightarrow}U\s{\overline e}{\rightarrow}V\s{\overline
f}{\rightarrow}TY$ , and $X\s{\overline
{da}}{\rightarrow}U\s{\overline g}{\rightarrow}W\s{\overline
h}{\rightarrow}TX$ be right triangles in $\cal Z / \cal D$. By the
proof for $TR(2)$ above, we may assume that $a$ and $d$ are ${\cal
D}-monic$. Then $da$ is also ${\cal D}-monic$. Hence we have the
following three right triangles in $\C$:
 $X\s{a}{\rightarrow}Y\s{b}{\rightarrow}Z\s{c'}{\rightarrow}TX$,
$Y\s{d}{\rightarrow}U\s{e}{\rightarrow}V\s{f'}{\rightarrow}TY$,
$X\s{da}{\rightarrow}U\s{g}{\rightarrow}W\s{h'}{\rightarrow}TX$. By
octahedral axiom, we have the following commutative diagrams in
$\C$:
$$ \begin{array}{ccccclclclcl}
X&\s{a}{\rightarrow}&Y&\s{b}{\rightarrow}&Z&\s{c'}{\rightarrow}&TX\\
\| &&\downarrow d&&\downarrow l&&\| &&\\
X&\s{da}{\rightarrow}&U&\s{g}{\rightarrow}&W&\s{h'}{\rightarrow}&TX\\
\ &&\downarrow e&&\downarrow i&&\ &&\\
&&V&=&V&&\\
\ &&\downarrow f'&&\downarrow &&\ &&\\
&&TY&\s{Tb}{\rightarrow}&TZ&&\\
\end{array}$$
and
\[ \begin{CD}
X@>da>>U @>g>>W@>h'>>TX\\
@VVaV@VV1 V @VViV @VVTa V \\
Y@>d>>U@>e>>V@>f'>>TY
    \end{CD} \]
We will show that $l$ is $\D -monic$. Let $j:Z\rightarrow D$ be any
morphism, where $ D \in \D$. Since $d$ is $\D -monic$, there exists
a morphism $k:U\rightarrow D$ such that $kd=jb$. Then $kda=jba=0$.
By Prop 2.2, there exists a morphism $m: W\rightarrow D$ such that
$mg=k$. Now we get $mlb=mgd=kd=jb,$ thus $(ml-j)b=0$. Then there
exists a morphism $\alpha :TX\rightarrow D$ such that $ml-j=\alpha
c'=\alpha h'l$. Then $j=(m-\alpha h')l$. Now by Lemma 3.8, we have
the following commutative diagrams in $\Z/\D$ where the rows are
right triangles:
$$ \begin{array}{ccccclclclcl}
X&\s{\overline a}{\rightarrow}&Y&\s{\overline b}{\rightarrow}&Z&\s{\overline c}{\rightarrow}&\sigma X\\
\| &&\downarrow \overline d&&\downarrow \overline l&&\| &&\\
X&\s{\overline {da}}{\rightarrow}&U&\s{\overline g}{\rightarrow}&W&\s{\overline h}{\rightarrow}&\sigma X\\
\ &&\downarrow \overline e&&\downarrow \overline i&&\ &&\\
&&V&=&V&&\\
\ &&\downarrow \overline f&&\downarrow &&\ &&\\
&&\sigma Y&\s{\sigma (\overline b)}{\rightarrow}&\sigma Z&&\\
\end{array}$$
and
\[ \begin{CD}
X@>\overline {da}>>U @>\overline g>>W@>\overline h>>\sigma X\\
@VV\overline aV@VV1 V @VV\overline iV @VV {\sigma}(\overline a)V \\
Y@>\overline d>>U@>\overline e>>V@>\overline f>>\sigma Y
    \end{CD} \]

    Therefore the quotient category $\Z/\D$ is a right triangulated category with shift functor
    $\sigma$.
\end{proof}

\begin{cor} Assume that $\C$, $\D$ and $\Z$ satisfy the same
conditions as in Theorem 3.9. Let $\mu \in Hom_{\cal Z}(M,N)$ be
$\cal D$$-monic$, where $M,N\in \Z$. If $M\s{\mu
}{\rightarrow}N\s{\gamma }{\rightarrow}P\s{\varphi}{\rightarrow} TM$
is a right triangle in $\cal C$, then $P\in \cal Z$.
\end{cor}

\begin{proof}
According to the proof of Theorem 3.9, for any morphism $\mu \in
Hom_{\cal Z}(M,N)$, there exists a right triangle
$M\s{{\mu}'}{\rightarrow}N\oplus D_M\rightarrow P'\rightarrow TM$
with ${\mu}'= \left(\begin{array}{cclcl}
{\mu}\\
{\alpha}_M
\end{array}\right)$ and $P' \in \cal Z$. Suppose $M\s{\mu }{\rightarrow}N\s{\gamma }{\rightarrow}P\s{\varphi}{\rightarrow} TM$
 is the right triangle in $\cal C$ containing $\mu$ as a part. We have the following commutative diagrams (compare the Definition
 3.5), where the rows are right triangles in $\C$:

$$ \begin{array}{ccccclcl}
M&\s{\mu}{\rightarrow} &N&\rightarrow &P&\rightarrow &TM\\
\| &&\downarrow  n'&&\downarrow g_1&&\| \\
M&\s{{\mu}'}{\rightarrow} &N\oplus D_M&\rightarrow &P'&\rightarrow &TM\\
\| &&\downarrow  f&&\downarrow g_2&&\| \\
M&\s{\mu}{\rightarrow} &N&\rightarrow &P&\rightarrow &TM,\\
\end{array}$$

where $n'= \left(\begin{array}{cclcl}
1_N\\
n
\end{array}\right)$ and $f=\left(\begin{array}{cccl}
1_N&0\\
\end{array}\right)$.
Since $fn'=1_N$, by Prop 2.13, we have that $g_2g_1$ is an
isomorphism. Thus P is a direct summand of $P'$. Therefore $P \in
\cal Z.$
\end{proof}
\bigskip

\subsection{Triangulated structure on the quotient category $\cal Z / \cal D$}

\begin{thm} Let $\C$ be a right triangulated category satisfying
Assumption 2.10. Assume that $\cal D$$\subseteq \cal Z$ are
subcategories of $\cal C$, $\cal D$ is factor-through-epic, $\cal Z$
is extension-closed and $(\Z,\Z)$ is a $\D-$mutation pair. If the
restriction of the shift functor $T$ to $\Z$, $T|_{\cal Z}:{\cal
Z}\rightarrow T\cal Z$ is full, then the right triangulated category
$\cal Z / D$ is a triangulated category.
\end{thm}

\begin{proof}
By Theorem 3.9, we only need to show that the shift functor $\sigma$
in $\Z/\D$ is an equivalence. Since $(\cal Z,Z)$ is a ${\cal
D}-$mutation pair, for any object $X \in \cal Z$, we fix a right
triangle in $\C$: $\omega X \s{\alpha ^X}{\rightarrow}D ^X\s{\beta
^X}{\rightarrow}X\s{\gamma ^X}{\rightarrow}T\omega X$ with $\beta^X$
being a right $\D-$approximation and $\alpha^X$ being a left
$\D-$approximation. For any morphism $f: X\rightarrow Y$, there
exist $g$ and $h$ such that $(g,f,h)$ is a morphism of right
triangles (i.e. the following diagram commutes):
\[ \begin{CD}
D^X@>\beta ^X >>X @>\gamma ^X>>T\omega X@>-T\alpha ^X >>TD^X\\
@VVgV@VVfV @VV hV @VV V \\
D^Y@>{\beta ^Y}>>Y@>{\gamma ^Y}>>T{\omega Y}@>-T\alpha ^Y>>T{D^Y}
    \end{CD} \]
By the fullness of the functor $T|_{\cal Z}$, there exists $h' \in
Hom_{\C}(\omega X,{\omega X}')$ such that $Th'=h$. Then we have the
following commutative diagram:
\[ \begin{CD}
\omega X@>\alpha ^X>>D^X@>\beta ^X >>X @>\gamma ^X>>T\omega X\\
@VVh'V@VVgV @VV fV @VV Th'V \\
\omega Y@>\alpha ^Y>>D^Y@>\beta ^Y>>Y@>\gamma ^Y>>T\omega Y.
    \end{CD} \] Now we define $\omega: {\cal Z / D}\rightarrow {\cal Z /
    D}$ as a functor which sends $X$ to $\omega X$, sends
    $\overline{f}$ to $\overline{h'}$. By using Lemma 3.2, one can prove that $\omega :{\cal Z /
D}\rightarrow {\cal Z / D}$ is an additive functor in a way dual to
the construction of $\sigma$ (compare Proposiotion 3.7 and its
proof).  Then we have the following commutative diagram in $\C$:
$$ \begin{array}{ccccclcl}
\omega X&\s{\alpha^X}{\rightarrow} &D^{X}&\s{\beta^X}{\rightarrow} &X&\s{\gamma ^X}{\rightarrow }&T\omega X\\
\| &&\downarrow  &&\downarrow g_X&&\| \\
\omega X&\s{\alpha_{\omega X}}{\rightarrow }&D_{\omega
X}&\s{\beta_{\omega X}}{\rightarrow} &\sigma \omega
X&\s{\gamma_{\omega X}}{\rightarrow }
&T\omega X\\
\| &&\downarrow  &&\downarrow g'_X&&\| \\
\omega X&\s{\alpha^X}{\rightarrow} &D^{X}&\s{\beta^X}{\rightarrow}
&X&\s{\gamma ^X}{\rightarrow }&T\omega X.\\
\end{array}$$
It follows from the commutative diagram that $\gamma ^X(1_X-g'g)=0$
and $\gamma_{\omega X}(1_{\sigma\omega X}-gg')=0$. By using Lemma
3.3, one can get that $\overline g_X \circ \overline
{g_X}'=\overline 1_{\sigma \omega X}$, $\overline {g_X}' \circ
\overline g_X=\overline 1_X$. Then $X \cong \sigma \omega X$ in
$\Z/\D$. For any morphism $f: X\rightarrow Y$, we also have the
following commutative diagram:
$$ \begin{array}{ccccclcl}

\omega X&\s{\alpha_{\omega X}}{\rightarrow }&D_{\omega
X}&\s{\beta_{\omega X}}{\rightarrow} &\sigma \omega
X&\s{\gamma_{\omega X}}{\rightarrow }
&T\omega X\\
\| &&\downarrow  &&\downarrow g'_X&&\| \\
\omega X&\s{\alpha^X}{\rightarrow} &D^{X}&\s{\beta^X}{\rightarrow} &X&\s{\gamma ^X}{\rightarrow }&T\omega X\\
\downarrow h'&&\downarrow  &&\downarrow f&&\downarrow \\
\omega Y&\s{\alpha^Y}{\rightarrow} &D^Y&\s{\beta^Y}{\rightarrow} &Y&\s{\gamma^Y}{\rightarrow} &T\omega Y\\
\| &&\downarrow  &&\downarrow g_Y&&\| \\

\omega Y&\s{\alpha_{\omega Y}}{\rightarrow }&D_{\omega
Y}&\s{\beta_{\omega Y}}{\rightarrow} &\sigma \omega
Y&\s{\gamma_{\omega Y}}{\rightarrow }
&T\omega Y.\\
\end{array}$$

Then we get the following commutative diagram

$$ \begin{array}{ccccclcl}

\omega X&\s{\alpha_{\omega X}}{\rightarrow }&D_{\omega
X}&\s{\beta_{\omega X}}{\rightarrow} &\sigma \omega
X&\s{\gamma_{\omega X}}{\rightarrow }
&T\omega X\\
\downarrow h'&&\downarrow  &&\downarrow f''&&\| \\
\omega Y&\s{\alpha_{\omega Y}}{\rightarrow }&D_{\omega
Y}&\s{\beta_{\omega Y}}{\rightarrow} &\sigma \omega
Y&\s{\gamma_{\omega Y}}{\rightarrow }
&T\omega Y.\\
\end{array}$$

where  $f''=g_Y \circ f \circ g'_X.$ Then we have that $\overline
{f''}=\sigma {\overline h'}=\sigma \omega {\overline f}$ in $\Z/\D$.
We also have the following commutative diagram in $\Z/\D$:
$$ \begin{array}{cccl}
X&\s{\overline g_X}{\rightarrow}&\sigma \omega X\\
\downarrow \overline f&&\downarrow \overline {{f'}'}\\
Y&\s{\overline g_Y}{\rightarrow}&\sigma \omega Y\\
\end{array}$$
Since  $\overline{g_X}$ and $\overline{g_Y}$ are isomorphisms in
$\Z/\D$,  $\sigma \omega$ is an equivalence. Dually one can show
that $\omega \sigma$ is also an equivalence. This proves $\sigma$ is
an equivalence.
\end{proof}

A direct application of Theorem 3.12 to the special case when $\C$
is a triangulated category is the following result.

\begin{cor}
Let $\cal C$ be a triangulated category, and $(\cal Z,Z)$ a ${\cal
D}-mutation$ pair. Then the quotient category $\cal Z / D$ is a
triangulated category.
\end{cor}

\bigskip

\subsection{Examples}

\begin{exm} In [IY], for a rigid subcategory $\cal D$ of a triangulated category $\cal C$, i.e. $\cal D$ satisfies
$Hom_{\C}({\cal D},T{\cal D})=0$, Iyama and Yoshino defined $\cal
D-$mutation pair in $\cal C$. It is easy to see that when $\cal D$
is rigid, and $(\cal X,Y)$ is a ${\cal D}-mutation$ pair, then
$(\cal X,Y)$ is a $\cal D-$mutation pair in the sense of Definition
2.6.
 It follows from Corollary 3.12 that the quotient category $\cal Z / \mathcal{D}$ is a triangulated category, which is Theorem 4.2 in [IY].
\end{exm}

\begin{exm}
Let $\cal C$ be a triangulated category with shift functor $T$,
$\cal X \subseteq \cal C$ is a functorially finite subcategory.
J{\o}rgensen proved in [J] that if $\cal C$ has Serre functor $S$
which satisfies $T^{-1} \circ S\cal X=\cal X$, then in the triangle
$M\rightarrow N\rightarrow P\rightarrow TP$, $M\rightarrow N$ is
$\cal X$$-monic$ if and only if $N\rightarrow P$ is $\cal X$$-epic$.
So $(\cal C,C)$ is a ${\cal X}-mutation$ pair defined in Definition
2.6. It follows from Corollary 3.12, the quotient category $\cal C /
\mathcal{D}$ is a triangulated category, which is the sufficient
part of  Theorem 3.3 in [J].
\end{exm}

\begin{exm} Let $\cal C$ be a $2-$Calabi-Yau triangulated category
[KR,IY], and $E$ an rigid object in $\cal C$, i.e. $Ext^1_{\cal
C}(E,E)=0$. The quotient category ${\cal {C}}/\mbox{add} E$ is a
right triangulated category [ABM]. Let $\cal X$ be the subcategory
of $\cal C$ consisting of objects $X$ with Hom$_{\cal C}(E,X[1])=0$.
It is easy to see that $(\cal X, \cal X)$ is an
$\mbox{add}E$-mutation pair in $\cal C$. Now passing to the right
triangulated quotient category
 ${\cal {C}}/\mbox{add} E$, we get the quotient subcategory of ${\cal {C}}/\mbox{add} E$,
denoted by $\cal Z$.  Then $(\cal Z,\cal Z)$ is a $0-$mutation pair
in ${\cal {C}}/\mbox{add} E$, hence $\Z$ is a triangulated
subcategory in $\cal C$ by Theorem 3.11.
\end{exm}

\begin{exm} Let $(\cal B, \cal S)$ be the exact category defined in [H]. Assume
that $(\cal B, \cal S)$ has enough $\cal S$-injectives, $\underline
{\cal{ B}}$ is the quotient category $\cal{B}/\cal{S}$
 [H]. Then according to the theorem of Chapter 1.2 in [H],
$\underline{\cal{B}}$ is a right triangulated category.
\medskip

{\bf Claim:} If all the $\cal S$-injectives are also $\cal
S$-projectives (which means that the set of $S-$injectives
 is contained in the set of $S-$projectives) , then the shift functor
 in $\underline{\cal{B}}$ is fully
and faithful. Moreover the right triangulated category
 $\underline{\cal B}$ satisfies the Assumption 2.10.\\
 {\bf Proof.}
 We will give a brief proof of this claim:\\
 Note that we will use the notations in [H], just changing "triangle" into "right triangle".\\
(i)For any morphism $f:X\rightarrow Y$ in $\cal B$, there exists a commutative diagram of short exact sequences\\
$$ \begin{array}{ccclclclclcl}
0&\s{0}{\rightarrow}&X&\s{\mu(X)}{\rightarrow}&I(X)&\s{\pi(X)}{\rightarrow}&TX&\s{0}{\rightarrow}&0\\
\ &&\downarrow f&&\downarrow I(f)&&\downarrow f'&&\ &&\\
0&\s{0}{\rightarrow}&Y&\s{\mu(Y)}{\rightarrow}&I(Y)&\s{\pi(Y)}{\rightarrow}&TY&\s{0}{\rightarrow}&0.\\
\end{array}$$
Then $\underline {f'}=T\underline f$ in $\underline {\cal B}$ by the
definition of $T$ [H]. Suppose $\underline {f'}=0$. Then $f'$
factors through some $\cal S$$-injective$ I. But I is also $\cal
S$$-projective$, then $f'$ factors through I(Y).
Let $f'=\pi(Y)\alpha$ with $\alpha:TX\rightarrow I(Y)$, then $\pi(Y)(I(f)-\alpha\pi(X))=0$,
which means $I(f)-\alpha\pi(X)$ factors through Y. Let $I(f)-\alpha\pi(X)=\mu(Y)\beta$,
where $\beta:I(X)\rightarrow Y$, then $\mu(Y)(f-\beta\mu(X))=0$. Since $\mu(Y)$ is monic, $f=\beta\mu(X)$, thus $\underline f=0$.
 This proves $T$ is faithful.\\
(ii)For any morphism $f':TX\rightarrow TY$, we have a commutative diagram of short exact sequences
$$ \begin{array}{ccclclclclcl}
0&\s{0}{\rightarrow}&X&\s{\mu(X)}{\rightarrow}&I(X)&\s{\pi(X)}{\rightarrow}&TX&\s{0}{\rightarrow}&0\\
\ &&\downarrow f&&\downarrow g&&\downarrow f'&&\ &&\\
0&\s{0}{\rightarrow}&Y&\s{\mu(Y)}{\rightarrow}&I(Y)&\s{\pi(Y)}{\rightarrow}&TY&\s{0}{\rightarrow}&0,\\
\end{array}$$
 where the morphism $g$ exists due to that I(X) is $\cal S$$-projective$.
  Now there exists a commutative diagram of short exact sequences
$$ \begin{array}{ccclclclclcl}
0&\s{0}{\rightarrow}&X&\s{\mu(X)}{\rightarrow}&I(X)&\s{\pi(X)}{\rightarrow}&TX&\s{0}{\rightarrow}&0\\
\ &&\downarrow f&&\downarrow I(f)&&\downarrow f''&&\ &&\\
0&\s{0}{\rightarrow}&Y&\s{\mu(Y)}{\rightarrow}&I(Y)&\s{\pi(Y)}{\rightarrow}&TY&\s{0}{\rightarrow}&0\\
\end{array}$$
where $\underline {f''}=T\underline f$. By $(g-I(f))\pi(X)=0$, we have that there exists a
morphism $\alpha:T(X)\rightarrow I(Y)$ such that $g-I(f)=\alpha\pi(X)$. Then $(f'-f'')\pi(X)=\pi(Y)(g-I(f))=\pi(Y)\alpha\pi(X)$.
Now we get $(f'-f''-\pi(Y)\alpha)\pi(X)=0$.
Since $\pi(X)$ is epic, $(f'-f''-\pi(Y)\alpha)=0$. Then $\underline f'= \underline {f''}=T\underline f$. This proves $T$ is full.\\
(iii)Assume that $Y\s{\underline g}{\rightarrow}Z\s{\underline
h}{\rightarrow}TX\s{-T\underline u}{\rightarrow}TY$ is a right
triangle in $\underline {\cal B}$. Then we can get the following
commutative diagram:
\[ \begin{CD}
TX@>-T\underline u >>TY @>-T\underline v >>TC_u@>-T\underline w >>T^2X\\
@VV1V@VV1V @VV \underline l'V @VV 1V \\
TX@>-T\underline u>>TY@>-T\underline g>>TZ@>-T\underline h>>T^2X\\
    \end{CD} \]
where the rows are right triangles, and $X\s{\underline
u}{\rightarrow}Y\s{\underline v}{\rightarrow}C_u\s{\underline
w}{\rightarrow}TX$ is a standard right triangle
 in $\underline
{\cal B}$. By Prop 2.13, $\underline l'$ is an isomorphism. Let
$\underline {l'k'}=1_{TZ}$ and $\underline {k'l'}=1_{TC_u}$. Since T
is full, there exist morphisms $\underline l:C_u\rightarrow Z$ and
$\underline k:Z\rightarrow C_u$ such that $T\underline l=\underline
l'$, $T\underline k=\underline k'$. Now we get $T\underline
{kl}=T1_{C_u}$ and $T\underline {lk}=T1_Z$. Since T is faithful, $l$
is an isomorphism. Then $X\s{\underline
u}{\rightarrow}Y\s{\underline g}{\rightarrow}Z\s{\underline
h}{\rightarrow}TX$ is a right triangle in $\underline {\cal B}$.
This proves that $\underline {\cal B}$ satisfies the Assumption
2.10.
\end{exm}

\begin{center}
\textbf {ACKNOWLEDGMENTS.}\end{center} The authors would like to
thank the referee for his/her very useful suggestions to improve the
paper.

\end{document}